\documentclass[11pt, a4paper]{article}
\usepackage{amsfonts}
\usepackage{amssymb}
\usepackage{amsmath}
\usepackage{amsthm}
\usepackage{mathabx}
\usepackage{epsfig,epic,eepic}

\usepackage[normalem]{ulem}

\newcommand{\R}{\mathbb R}

\newcommand{\Z}{\mathbb Z}

\newcommand{\hop}{\vskip .2cm\noindent}

\newcommand{\mrm}{\mathrm}

\newcommand{\ms}{^{-1}}
\newcommand{\alb}{{\alpha\beta}}
\newcommand{\bb}{\mathcal B}

\newtheorem{enonce}{}[section]

\newtheorem{theo}[enonce]{Theorem}
\newtheorem{cor}[enonce]{Corollary}
\newtheorem{prop}[enonce]{Proposition}
\newtheorem{lem}[enonce]{Lemma}

\newtheorem{fact}[enonce]{Fact}

\theoremstyle{definition}
\newtheorem{defi}[enonce]{Definition}
\newtheorem{nota}[enonce]{Notation}

\theoremstyle{remark}
\newtheorem{rema}[enonce]{Remark}

\newtheorem{examples}[enonce]{Examples}
\title{Conformal class of Lorentzian surfaces  with Killing fields.}
\author{Pierre Mounoud}
\date{}
\begin{document}
 \maketitle
\begin{abstract}
We study the conformal classes of $2$-dimensional Lorentzian tori with (non zero) Killing fields. We define a map that associate to such a class a vector field on the circle (up to a scalar factor). This map is not injective but has finite dimensional fiber.
It allows us to characterize the conformal classes of tori with Killing field satisfying a condition  related to the existence of conjugate points  given 
by L.\ Mehidi. 
\end{abstract}
\section{Introduction}
It is well known that there is no Uniformisation Theorem for Lorentzian surfaces, ie something that would say that the space of Lorentzian  conformal structures (ie metrics up to conformal factor and up to diffeomorphism)  on a given surface is, for example, finite dimensional.  Indeed, two Lorentzian metrics on a surface are conformal if and only if they have the same lightlike cones. On an orientable surface these cones (made of two lines) define two foliations and it is known that there is no reasonable moduli space of foliations on a surface. One of the goals of this article is to better understand the set of conformal structures of Lorentzian tori admitting a non zero Killing field (in what follows Killing field will always mean  non zero Killing field). The motivation is coming from the desire to find a conformal interpretation of a Theorem by L.\, Mehidi in \cite{LM} (cf. Theorem \ref{thM}) about Lorentzian tori without conjugate points.

It is not difficult to see that a torus $T$ with a Killing field $K$ is conformal to a flat metric (ie has linear lightlike foliations) if and only if all the orbits of $K$ have the same type (between spacelike, lightlike and timelike)  otherwise said  if the sign of $\langle K,K\rangle$ is  constant. Hence, we will assume now that the sign of  $\langle K,K\rangle$ is not constant. In such a case the flow of $K$ has to be periodic and the space of orbits of $K$ is a circle. 
In \cite{BetM}, Ch.\,Bavard and the author gave a natural parameterization of this circle and therefore it is possible to associate to $(T,K)$ a well defined smooth periodic real  function $f$ induced by $\langle K,K\rangle$.  
Contrarily to the Riemannian case, $\widetilde T$, the universal cover of $T$, is not determined by $f$. The function $f$ determines a bigger space $E_f$  whose construction is recalled in section \ref{sect_Ef}.  This space $E_f$ contains $\widetilde T$, otherwise said $T$ is locally modeled on $E_f$. In order to find  $\widetilde T$ in $E_f$ (ie to define the developping map) it is necessary to add a finite combinatorial data that precises to which lightlike foliation belong the  lightlike orbits of $K$.  
We denote by Per$(f)$ the set of non zero periods of $f$ and if $P\in \mathrm{Per}(f)$ we denote by  $X_{f,P}$ the vector field on $\R/P\Z$ induced by $f\partial_t$. 
In the same way as $T$ is almost determined by $f$, the conformal class of $T$ is almost determined by a vector field $X_{f,P}$. More precisely:
\begin{theo}\label{cordestores}
Let $f$ and $g$ be two smooth periodic functions of non constant sign. Let $T$ be a torus locally modeled on $E_f$.
There exists a torus locally modeled on $E_g$ that is  conformal to a finite cover of $T$  if and only if there exists  $P\in \mathrm{Per}(f), Q\in \mathrm{Per}(g)$ and $a\neq 0$  such that $X_{f,P}$ is diffeomorphic $aX_{g,Q}$.
\end{theo}
Moreover, see Remark \ref{rem_u}, two conformal tori locally modeled on the same space $E_f$ are isometric. As, by \cite[Proposition 5.2]{BetM}, the set of tori locally modeled on $E_f$ is finite dimensional, we see that  the fibers of the map that sends the conformal class of a torus $T$  locally modeled on $E_f$ on the  vector fields $\R.X_{f,P}$, with $P$ the period of $f$ deduced from $T$, are finite dimensional.

This theorem, together with Mehidi's Theorem from \cite{LM}  mentioned above and the classification of hyperbolic vector fields on the circle,  implies the following.
\begin{cor}
Let $T$ be a non flat Lorentzian torus with a Killing field. If $T$ is without conjugate point and if the closed lightlike geodesics of $T$ are incomplete then there exists a unique $b\in ]-1,1[$ such that $T$ is in the conformal class of a Reeb torus locally modeled on $E_{sin(y)(1+b\sin(y))}$.
\end{cor}
Note that a Lorentzian torus is Reeb if its lightlike foliations are unions of Reeb components. A Reeb torus locally modeled on a space $E_f$ is determined by the choice of a period of $f$, of the length of the orbits of the Killing field and of a twist parameter, see \cite{BetM} for details.
The Clifton-Pohl tori correspond to the case $b=0$. When $|b|<1/8$, it follows from \cite{LM} that Reeb tori locally modeled on $E_{sin(y)(1+b\sin(y))}$ do not have conjugate points. The author does not know if the other conformal classes actually contains  metrics without conjugate points.
%
\hop
{\bf Acknowledgment} The author thanks  Nicolas Gourmelon for helpfull discussions.
\section{Conformal geometry of  ribbons}\label{sect_ribb}
Let $I$ be an (open) interval and $f:I\rightarrow \R$ be a 
 smooth function with non constant sign. A point  $x\in I$ such that  $f(x)=0$ and $f'(x)\neq 0$ will be called a simple zero of $f$.

 We  call the surface $\R \times I$ endowed with the Lorentzian metric $ f(y)dx^2+2 dxdy$ the \emph{Ribbon associated to $f$} and denote it $R_f$. The vector field $\partial_x$ is clearly a Killing field of $R_f$, also denoted~$K_f$. 
\begin{rema}
Let $\Sigma$ be a Lorentzian surface and if $K$ be a (non zero) Killing field of $\Sigma$. Any $p\in \Sigma$ such that $K(p)\neq0$ has a  neighborhood isometric to a ribbon. Indeed, let $\gamma$ be a (maximal) lightlike geodesic transverse to $K$ and containing  $p$, its $K$-saturation, ie its image by the flow of $K$, is a ribbon (because the map $t\mapsto \langle \gamma'(t),K(\gamma(t))\rangle$ is constant). The map $f$ is given by $t\mapsto \langle K(\gamma(t)),K(\gamma(t))\rangle$). It follows that $\Sigma$ can be covered by ribbons and neighborhoods of saddle points.
\end{rema}
\begin{defi}
Let $\{I_\alpha,\alpha\in \bb\}$ be the set of connected component of $I\smallsetminus f\ms(0)$. The order of $\R$ induces an order denoted $\prec$ on $\bb$ (more precisely $\alpha\prec \beta$ if $\forall (y,y') \in I_\alpha\times I_\beta$ we have $y<y'$).
\begin{enumerate}
\item For any $\alpha\in \bb$, the submanifold $B_\alpha=\R\times  I_\alpha$ of $R_f$ is called an open strip of $R_f$. 
\item If $J$ is a maximal subinterval of $I$ containing  a unique (simple) zero of $f$  then $\R\times  J$  is called a (simple) domino of $R_f$. 
\end{enumerate}
\end{defi}
On each interval $I_\alpha$ we choose a primitive $F_\alpha:I_\alpha\rightarrow \R$ of $-\frac1f$. If $\alpha$ is not the smallest or biggest element of $\bb$ then $F_\alpha$ is  clearly a diffeomorphism. 
We easily verify that for any $\alpha\in \bb$, the curve $\gamma_\alpha: t\mapsto  (F_\alpha (t),t)$ is a (pre)geodesic perpendicular to $K$ contained in $B_\alpha$. 

The  orthogonal reflection relatively to $\gamma_\alpha$, ie the map $\sigma_\alpha: B_\alpha \rightarrow B_\alpha, (x,y)\mapsto (2F_\alpha(y)-x,y)$, is an isometry that does not extend to $R_f$.  We  will call it the  \emph{generic reflection of $R_f$ with axis $\gamma_\alpha$.}  It sends the lightlike geodesic $\{x=0\}$ on  the curve parameterized by $t\mapsto  (2 F_\alpha (t),t)$ which is therefore a lightlike geodesic.

In \cite{BetM} the generic reflections are used to extend $R_f$. Indeed, if we glue two copies of $R_f$ along a strip $B_\alpha$ thanks to a generic reflection $\sigma_\alpha$ we obtain a new Lorentzian surface.
\begin{defi}
The vertical  foliation of $R_f$ is the lightlike foliation whose leaves are the lines $\{x=c; c\in \R\}$. 
\\
The horizontal  foliation of $R_f$ is the foliation whose leaves are the lines $\{y=c; c\in f\ms(0)\}$ and the curves parameterized by $I\rightarrow \R_f, t\mapsto  (2F_\alpha(t)+x,t)$, for any $\alpha\in \bb$ and $x\in \R$.
\end{defi}
\begin{defi}
Two Lorentzian surfaces with a Killing field $S_1$ and $S_2$ are $K$-conformal if there exists  a conformal diffeomorphism between them that sends Killing fields of $S_1$ on Killing fields of $S_2$.  
\end{defi}
\begin{rema} It is not difficult to see that a "$K$-isometry" between ribbons must read $(x,y)\mapsto (ax+t_0, y/a+b)$. Consequently, two ribbons $R_f$ and $R_g$ are $K$-isometric if and only if  there exists $a\neq 0, b\in \R$ such that $g(y)=a^2f(y/a+b)$.  Following \cite{BetM}, we say that $f$ and $g$ are in the same class and write  $[[f]]=[[g]]$.
\end{rema}
Recall that a diffeomorphism between orientable Lorentzian surfaces is conformal if and only if it sends the lightlike foliations of the first on the lightlike foliations of the second.
\begin{prop}\label{prop_rib}
Let $f:I\rightarrow \R$ and $g:J\rightarrow \R$ be two smooth maps and let $R_f$ and $R_g$ be the associated ribbons. 
A map  $\Psi_{f,g}: R_f\rightarrow R_g$ is a $K$-conformal diffeomorphism if and only if  $\Psi_{f,g} (x,y)=(a x+t_0,\varphi(y))$ where $t_0\in \R$, $a\in \R^*$   and $\varphi:I\rightarrow J$ is a diffeomorphism such that $\varphi_*\,f(t)\partial_t=ag(t)\partial_t$.
\end{prop}
\begin{proof}
If $\Psi_{f,g}$ is $K$-conformal it sends the vertical foliation of $R_f$ on the vertical foliation of $R_g$ and $\partial_x$ on $a\partial_x$ for some $a\neq 0$. Therefore it must read $(x,t)\mapsto (ax+t_0, \varphi(y))$ for some $t_0\in \R$ and some diffeomorphism $\varphi$. But it also preserves the  foliations orthogonal to the Killing fields. Their tangent spaces are respectively spaned by $\partial_x-f(y)\partial_y$ and $\partial_x-g(y)\partial_y$, therefore $d_{(x,y)}\Psi_{f,g}(\partial_x-f(y)\partial_y)=(a\partial_x-\varphi'(y)f(y)\partial_y=a(\partial_x-g(\varphi(y))\partial_y)$ ie $\varphi_*f\partial_y=g\partial_y$.
The reciprocal is  clear.
\end{proof}
Abusing notations we will denote by $\Psi_{f,g}^*R_g$ the surface $\R\times I$ endowed with  the pull-back of the metric of $R_g$. Note that, $\Psi_{f,g}^*(g(y)dx^2+ 2dxdy)=a\varphi'(y)(f(y)dx^2+2dxdy)$. In particular, if the vector field $f\partial_t$ is complete, for any $t\in \R$ we can consider $\varphi_t$ its flow at time~$t$  and  define  the map $\Psi_{f,f}:(x,y)\mapsto (x,\varphi_t(y))$ that is a $K$-conformal diffeomorphism of~$R_f$.  

Note that the foliations perpendicular to $\partial_x$ are the same for both metrics and therefore the generic reflections of $R_f$ are the generic reflections of $\Psi_{f,g}^*R_g$.  

\section{Conformal geometry of  saddles} \label{sect_saddles}

Let $J$ be an interval containing $0$ and $\theta:J\rightarrow \R_+^*$  be a smooth function. The surface  $\{(u,v)\in \R^2\,;\, uv\in J\}$ endowed with the metric $2\theta(uv) du dv$ is called the symmetric saddle\footnote{The definition of symmetric saddle given in \cite{BetM} is a priori more general but it follows from Proposition  \ref{prop_ext} that these definitions are  in fact equivalent} associated to $\theta$ and is denoted $S_\theta$. The vector field  $u\partial_u- v\partial_v$. is a Killing field of $S_\theta$, we will denote it by $K_\theta$. The geodesics perpendicular to $K_\theta$ are the radial lines, they all meet at the origin,  and the generic reflections they define are global isometries of $S_\theta$ (they read $(u,v)\mapsto (\pm e^tv, \pm e^{-t}u)$ where $t$ is a real parameter).

If $\gamma$ is a horizontal or vertical line that does not contain $0$ then its $K_\theta$-saturation, ie $\bigcup_{t\in\R}\Phi^t_K(\gamma)$ where $\Phi_K$ is the flow of $K$, is a half-plane that we will call a half-saddle. Consequently, see \cite[Lemme 2.3]{BetM}, each of the four half-saddles is isometric to a domino and $S_\theta\smallsetminus\{0\}$ is the union of 4 dominos. The generic reflections permute the dominos which are therefore all  isometric to a given simple domino $D_f$. More precisely, if $\gamma: t\mapsto (u(t),1)$ is the geodesic starting from $(0,1)$ such that $\langle \gamma'(t),K_\theta(\gamma(t)\rangle\equiv 1$, ie $u$ is the solution of  $u'=-\frac 1{\theta(u)}$, $u(0)=0$,   then we can take  $f=-2u(\theta\circ u)$ (we have  normalized the choice of $f\in[[f]]$ by asking  that $f(0)=0$ and $f'(0)=2$). 

The following result is already contained is \cite{BetM}. The following proof of it is simpler than the original one but chiefly it will allow us to prove Proposition \ref{prop_confsad} which  is the goal  of this section. 
\begin{prop}\label{prop_ext}
Any simple domino $D_f$ can be  isometrically embedded as a half-sadlle in a unique  symmetric saddle $S_{\theta_f}$. Moreover, this embedding is unique up to left composition by isometry of $S_{\theta_f}$.
\end{prop}
\begin{proof} 
We start with the linear case ie with the domino $D_0=\R\times I_0$ defined by  $f_0(y)=\lambda y$, $\lambda\in \R^*$,  $I_0\ni 0$ but   may be different from $\R$. In this case, an embedding is given by $\Phi_0: (x,y)\mapsto (y\exp(\lambda x/2), \exp(-\lambda x/2))$ as one can check that $\Phi_0\ms{}^*( 2dxdy + \lambda y dx^2)=-\frac 1\lambda 2dudv $ and that $\Phi_0{}_*\partial_x= \frac \lambda 2 (u\partial_u-v\partial_v)$.

Let $f:I\rightarrow \R$ be a smooth function with a unique simple zero (wlog we assume that $f(0)=0$).  The vector field $f(y)\partial_y$ on $I$ is \emph{hyperbolic} (ie the zeros of $f$ are simple) with a unique $0$. It is well known that it  can  be linearized, see for example  \cite[Theorem 2.16]{BelTrak}, ie there exists a diffeomorphism $\varphi: I\rightarrow I_0$ such that $\varphi_*(f(t)\partial_t)=\lambda t\partial_t$, $\lambda=f'(0)$ ($I_0$ depends on the completeness of $f(y)\partial_y$). 
According to Proposition  \ref{prop_rib}, there exists a conformal   diffeomorphism $\Psi_f: D_f \rightarrow D_0$ sending $K_f$ on $K_{f_0}$, otherwise said there exists a  function $\zeta$ such that $\Psi_f\ms{}^*(2dxdy +f(y) dx^2)=\zeta(y) (2dxdy + \lambda y dx^2)$. The composition $\Phi_f=\Phi_0\circ \Psi_f$ is an isometric embedding of $R_f$ into the symmetric saddle $S_{\theta_f}=(\{(u,v)\in \R^2\,;\, uv\in I_0\}, -\frac {\zeta(uv)} {\lambda}\, 2dudv)$ sending $K_f$ on $\frac \lambda 2 (u\partial_u-v\partial_v)$.

If $\Xi:D_f\rightarrow S_{\theta_f}$ is another isometric embedding then its image is a half-space of $S_f$, therefore there exists a generic reflection $\sigma$ such $\sigma\circ \Xi (D_f)=\Phi_0\circ \Psi_f (D_f)=\Omega^+=\{(u,v)\in \R^2\,;\, uv\in I_0, \ v>0\}$. Hence, $(\sigma\circ \Xi) \circ (\Phi_0\circ \Psi_f)\ms$ is an isometry of $\Omega^+$, but every isometry of $\Omega^+$ results from the flow of $K$, \cite [Proposition 2.6]{BetM}, and therefore is the restriction of an isometry of $S_{\theta_f}$. Proving the uniqueness of the embedding in $S_{\theta_f}$. 

It follows that if $D_f$ embedds in $S_{\theta'}$ then $\theta'$ and $\theta_f$ coincide on a half-saddle and therefore are equal.
\end{proof}
\begin{defi}
Let $D_f$ be a simple domino. Let $\sigma_\alpha$ and $\sigma_\beta$ be two generic reflections of $D_f$ with axis $\alpha$ and $\beta$.  We will say that $\sigma_\alpha$ and $\sigma_\beta$ are \emph{compatible} if there exists an isometrical embedding $\Psi$ of $R_f$ into a  symmetric saddle such that the geodesics containing  $\Psi(\alpha)$ and $\Psi(\beta)$ are orthogonal (at the saddle point).
\end{defi}
\begin{rema}
According to Proposition \ref{prop_ext}, any isometrical embedding into a symmetric saddle will send the axis of compatible generic reflections on orthogonal geodesics.
\end{rema}
Let $\sigma_\alpha$ and $\sigma_\beta$ be two compatible generic reflections on a simple domino $D_f$.  Let $D_1,\dots, D_4$ be 4 copies of $D_f$ and $\Sigma_f$ be the Lorentzian surface obtained by  gluing $D_1$ to $D_2$ thanks to $\sigma_\alpha$, $D_2$ to $D_3$ thanks to $\sigma_\beta$, $D_3$ to $D_4$ thanks to $\sigma_\alpha$ and finally $D_4$ to $D_1$ thanks to $\sigma_\beta$.
Let $\Psi_f$ be an isometric embedding of $D_f$ into the symmetric saddle $S_{\theta_f}$.  The local isometries  $\Psi_f\circ \sigma_\alpha\circ \Psi_f\ms$ and $\Psi_f\circ \sigma_\beta \circ \Psi_f\ms$ extends to a global isometry denoted by $\tilde\sigma_\alpha$ and $\tilde \sigma_\beta$.  We immerse  $\coprod _{1\leq i\leq 4}D_i$ in $S_{\theta_f}$ using  $\Psi_f$ for $D_1$,  $\tilde\sigma_\alpha\circ \Psi_f$  for $D_2$,  $\tilde \sigma_\beta\circ \tilde\sigma_\alpha\circ \Psi_f$  for $D_3$, and $\tilde\sigma_\alpha\circ \tilde \sigma_\beta\circ \tilde\sigma_\alpha\circ \Psi_f$  for $D_4$. 
As  $(\tilde\sigma_\beta\circ \tilde \sigma_\alpha)^2=\mathrm{id}$, because their axis are orthogonal, this immersion induces an isometry $\Lambda_f$  between $\Sigma_f$ and $S_{\theta_f}\smallsetminus \{0\}$ (or equivalently a $K$-conformal embedding in the flat saddle $S_0$)
\begin{fact}\label{factfun}
If $\Psi: D_f\rightarrow D_g$ is a  $K$-conformal diffeomorphism between simple dominos then $\Psi^*D_g$ and $D_f$ have the same pairs of \emph{compatible} generic reflections
\end{fact}
\begin{proof}
We keep the notation from the proof of Proposition \ref{prop_ext}. 
The map $\Phi_0\circ \Psi_f$ is a $K$-conformal embedding of both $\Psi_{f,g}^*D_g$ and $D_f$ in the flat saddle $S_0$. Therefore being compatible for $\Psi_{f,g}^*D_g$ or $D_f$ is the same.
\end{proof}
The above construction starting  with $D_g$, using the embedding $\Psi_g=\Psi_f\circ \Psi_{f,g}\ms$ of $D_g$ in $S_0$,  and the compatible (according to Fact \ref{factfun}) generic reflections $\Psi_{f,g}\circ \sigma_\alpha \circ \Psi_{f,g}\ms$ and $\Psi_{f,g}\circ \sigma_\beta \circ \Psi_{f,g}\ms$, provides us a surface $\Sigma_g$ and a conformal  embedding $\Lambda_g$ of $\Sigma_g$ in 
the flat saddle $S_0$. 
It is obvious that the map $\Psi_{f,g}$ induces a conformal diffeomorphism $\Lambda_{f,g}$ between $\Sigma_{f}$ and $\Sigma_{g}$ that satisfies  $\Lambda_g\circ \Lambda_{f,g}=  \Lambda_f$. It proves the following.
\begin{prop}\label{prop_confsad}
Let $D_f$ and $D_g$ be two simple dominos and let $S_{\theta_f}$ and $S_{\theta_g}$ be their saddle extensions.
Any $K$-conformal diffeomorphism $\Psi: D_f\rightarrow D_g$ is the restriction of a
$K$-conformal diffeomorphism between the saddles $S_{\theta_f}$ and $S_{\theta_g}$. 
\end{prop}
\section{The universal extensions}\label{sect_Ef} 
We recall in this section the construction of the so-called universal extension $E_f$ associated to a real function $f$. It is slightly modified in order to obtain more explicit atlases what simplifies the construction of maps between these spaces. 
Some details of the construction are nevertheless  left to the reader and can be found in \cite{BetM}.

One of the \emph{raison d'être} of these extensions is that if $\Sigma$ is a  surface with a Killing field whose ribbons can be embedded isometrically in a given ribbon $R_f$ then $\Sigma$ is locally modeled on $E_f$ (in the sense of $(G,X)$-structures). According to \cite[Proposition 3.8]{BetM}, if $\Sigma$ is connected and compact or real analytic  then it is is locally modeled on $E_f$ for some function  $f$.

We first fix some notations. Let $\{B_\alpha; \alpha\in \mathcal B\}$ be the set of strips of the ribbon $R_f$. 
 Let  $\mathcal C$ be the set of pairs   $\{\alpha,\beta\}$ of elements of $\bb$  such that  the strips $B_\alpha$ and $B_\beta$ are separated by a simple zero of $f$, we will say that such strips are \emph{contiguous}.
We denote by $D_{\alpha\beta}$ the domino containing these contiguous strips and by $S_{\alpha\beta}$ the symmetric saddle associated to it. 
Let $G$ be the group with the following presentation:
$$G=\langle  \mathcal B\ |\ \forall \alpha\in  \mathcal B, \alpha^2=1\  ; \   \forall (\alpha,\beta) \in\mathcal C, (\alpha\beta)^2=1 \rangle.$$
On each $B_\alpha$ we choose a generic reflection $\sigma_\alpha$. We  choose them such that if $(\alpha,\beta)\in \mathcal C$ then $\sigma_\alpha$ and $\sigma_\beta$ are compatible.  

We can now start the construction. 
Let $X=\coprod_{i\in G} R_i$, where the $R_i$  are copies of $R_f$.  We denote by $B_{i,\alpha}$ the copy on $R_i$ of the strip $B_\alpha$ of $R_f$ (more generally the copy on $R_i$ of any object $Z_*$ defined on $R_f$, will be denoted $Z_{i,*}$). For any $(i,j)\in G^2$ such that $\alpha:=i\ms j \in S$  we identify $B_{i,\alpha}$ to $B_{j,\alpha}$ thanks to $\sigma_\alpha$ (more correctly thanks to $\sigma_{i,\alpha}: B_{i,\alpha}\rightarrow B_{i\alpha,\alpha}, (x,y)\mapsto \sigma_\alpha(x,y)$). In doing so, we obtain a connected Lorentzian surface $Y_0$ with a Killing field (that reads $\pm \partial_x$ in the $R_i$ depending on the parity of the length of the word $i$ in $G$). 

We note that for every  $\{\alpha,\beta\}\in \mathcal C$  and every  $i\in G$, we have glued 
the dominos $D_{i,\alpha\beta}$, $D_{i\alpha,\alpha\beta}$, $D_{i\alpha\beta,\alpha\beta}$ and $D_{i\alpha\beta\alpha,\alpha\beta}$ into a surface isometric to $S_\alb\smallsetminus\{0\}$ (because we have chosen $\sigma_\alpha$ and $\sigma_\beta$ compatible, cf section \ref{sect_saddles}).
 Thus, we can glue a copy of $S_{\alpha\beta}$ to $Y_0$ along it in order to add the missing saddle point. Doing so systematically,  we obtain a Lorentzian surface $Y_f$ having a non trivial complete Killing field still denoted $K_f$.
\begin{fact}\label{factsc}
The surface $Y_f$ is simply connected.
\end{fact}
\begin{proof}
As the ribbons are simply connected, any loop $\ell$ in $Y$, is homotopic to a broken lightlike geodesic $\ell_0$. Let $w$ be the word in the alphabet $\mathcal B$ given by the index of the strips containing the breaking points of $\ell_0$ (we can suppose that no breaking point of $\ell_0$ is a saddle point). Since $\ell$ is a loop, the image of $w$ in $G$ has to be trivial. The strips and the  saddles being simply connected we can suppose that $w$ contains no $\delta^2$ with $\delta\in \mathcal B$ and no $(\alpha\beta)^2$ with $\{\alpha,\beta\}\in \mathcal C$.  It means that the image of $w$ in the free group generated by $\mathcal B$ is trivial and $\ell$ is therefore contractible.\end{proof}
\begin{prop}
The surface $Y_f$ is isometric to the universal extension $E_f$ defined in \cite{BetM} (and will be denoted $E_f$  now).
\end{prop}
\begin{proof}
It follows from \cite[Th\'eor\` eme 3.21]{BetM}. Indeed, $Y$ is "de classe $ [f] $ "  as any point of $Y$ that is not a saddle point is contained in a copy of $R_f$ and "r\'eflexive" as each strip of $Y$ is the intersection of two copies of $R_f$. Moreover, $Y$ is "sans selles à l'infini" (without saddles at infinity) ie every  lightlike orbits of  $K$ is contained in a complete lightlike orbit (because of the presence of the saddles). At last, $Y$ is simply connected according to fact \ref{factsc}.
\end{proof}
\begin{lem}[Lemmes 3.10 and 3.17 from \cite{BetM}]\label{lem_gen}
Any generic reflection (associated to $K_f$) of $E_f$ extends into a global isometry of $E_f$.

If $f$ is $T$-periodic, then there exits an  isometry $\tau$ of $E_f$ such that $\tau(R_1)=R_1$ and that the restriction of $\tau$ to $R_1$ reads $(x,y)\mapsto (x,y+T)$.
\end{lem}
\begin{proof}
According to \cite[Proposition 2.38]{BetM} any isometry between punctured saddles extends to an isometry between the saddles. Consequently, it is enough to prove the lemma on the surface $Y_0$. It is also enough to prove the first point  for the generic reflections used in the construction of $Y_0$ (and conjugate by the flow of $K$ to obtain the others). For any $(x,y)\in R_f$ and $i\in G$, we denote by $(x,y)_i$ the point of $R_i$ whose coordinates are $(x,y)$. 

Let $\alpha\in \bb$ and $i_0\in G$. The map $\widetilde 
\sigma_{\alpha}:\coprod R_i\rightarrow \coprod R_i, (x,y)_i\mapsto  (x,y)_{\alpha i}$ is an isometry that induces an isometry $\widehat \sigma_{\alpha}$ of $Y_0$ (because  $i\ms j=(\alpha i)\ms (\alpha j)$ and the $\sigma_\beta$'s are involutions). The image of the strip $B_{i_0,\alpha}$ in $Y_0$ is invariant by $\widehat \sigma_{\alpha}$ and coincide with $\sigma_{\alpha}$ on it.

Assume $f$ is $T$-periodic. The map $\tau_0:R_f\rightarrow R_f, (x,y)\mapsto (x,y+T)$ sends strips on strips, therefore there exists  a map $\mu:\bb\rightarrow \bb$  such that $\tau_0(B_\alpha)=B_{\mu(\alpha)}$. Clearly, if $\{\alpha,\beta\}\in \mathcal C$ then so does $\{\mu(\alpha),\mu(\beta)\}$. It means that $\mu$ preserves the relations defining $G$ and therefore induces an endomorphism  of $G$, still denoted $\mu$.  Moreover if $\sigma_\alpha$ and $\sigma_\beta$ are compatible then so are $\tau_0\circ \sigma_\alpha \circ \tau_0\ms$ and  $\tau_0\circ \sigma_\beta \circ \tau_0\ms$. Consequently, we can choose the generic reflections used to construct $E_f$ so that for any $\delta\in \bb$ , $\tau_0\circ \sigma_\delta \circ \tau_0\ms=\sigma_{\mu(\delta)}$.

Let $\widetilde \tau :\coprod R_i\rightarrow \coprod R_i, (x,y)_i\mapsto  (x,y+T)_{\mu(i)}$. Again $\widetilde \tau$ induces a map $\widehat\tau: Y^0\rightarrow Y^0$. Indeed, if $(x,y)_i$ and $(x',y')_j$ are  distinct points glued together, then $i\ms j=\alpha\in \bb$ and $(x',y')=\sigma_\alpha(x,y)$. Furthermore, $\widetilde \tau ((x,y)_i)=(x+T,y)_{\mu(i)}$ and $\widetilde \tau ((x',y')_j)=(x'+T,y')_{\mu(j)}$ but we see that $\mu(i)\ms\mu(j)=\mu(\alpha)$  
and  $\sigma_{\mu(\alpha)}(x'+T,y')=(x+T,y)$, therefore $\widetilde \tau ((x,y)_i)$ and $\widetilde \tau ((x',y')_j)$ are also identified on $Y^0$. Hence, $\widehat \tau$ is the isometry we are looking for.
\end{proof}
\begin{defi}
We denote by $\mathrm{Is}_{\rm gen}(E_f)$ the subgroup of the isometry group of $E_f$ generated by the generic reflections (associated to $K_f$). Note that $\mathrm{Is}_{\rm gen}^0(E_f)$, the identity component of  $\mathrm{Is}_{\rm gen}(E_f)$, is  the flow of $K_f$. 
We denote by $\mathrm{Is}^\pm(E_f, K_f)$ the group of isometries of $E_f$ sending $K_f$ on $\pm K_f$ (when the curvature is not constant, it is  the isometry group of $E_f$).
\end{defi}
\section{The results}
\subsection{Conformal geometry of the universal extensions}
\begin{nota}
Let $f:I\rightarrow \R$ be a smooth periodic function and Per$(f)$ be the set of its non trivial periods. For any $P\in \mathrm{Per}(f)$, we denote $X_{f,P}$ the vector field on $\R/P\Z$ induced by $f\partial_t$.
\end{nota}
\begin{theo}\label{theo_conf}
Let $f: I \rightarrow \R$, $g: J\rightarrow \R$   be two smooth  non constant functions. Let $E_f$ and $E_g$ be the universal extensions  associated to them.

There exists a conformal  diffeomorphism $\Phi : E_f\rightarrow E_g$ such that $\Phi\, \mathrm{Is}_{\rm gen}(E_f)\,\Phi\ms=\mathrm{Is}_{\rm gen}(E_g) $ if and only if there exists $a\in \R^*$ such that the vector fields $f\partial_t$ and $ag\partial_t$ are diffeomorphic.

Moreover, if $f$ and $g$ are both periodic and if 
there exists $(P,Q)\in \mathrm{Per}(f)\times \mathrm{Per}(g)$ such that $X_{f,P}$ and $X_{g,Q}$ are diffeomorphic 
then  we can choose $\Phi$ such that $\Phi\,
 \mathrm{Is}^\pm(E_f, K_f)\,\Phi\ms$ and $\mrm{Is}^\pm(E_g, K_g)$ are commensurable.  It is also true if $f$ and $g$ are both non periodic.
\end{theo}
Recall that two subgroups of a given group  are commensurable if their intersection has finite index in both of them.
\begin{proof} We first assume that there exists $a\in \R^*$ such that the vector fields $f\partial_t$ and $ag\partial_t$ are diffeomorphic.
According to Proposition \ref{prop_rib}, there exists a conformal  diffeomorphism $\Psi_{f,g}: R_f\rightarrow R_g$ sending $K_f$ on $a K_g$. Possibly replacing $g$ by another element in $[[g]]$ we can suppose that $a=1$. This diffeomorphism sends strips of $R_f$ on strips of $R_g$, respecting contiguity. For any $\alpha\in \bb$, we denote by $B'_\alpha$ the strip $\Psi_{f,g}(B_\alpha)$ and $\sigma'_\alpha$ the generic reflection $\Psi_{f,g}\circ \sigma_\alpha \circ \Psi_{f,g}\ms$.  According to Fact \ref{factfun}, if $\{\alpha,\beta\}\in \mathcal C$ then $\sigma'_\alpha$ and $\sigma'_\beta$ are compatible. It implies that the group used in the construction of $E_g$ is also $G$.  For any $i\in G$ we denote by $R'_i$ a copy of $R_g$.  We can  use the $\sigma'_\alpha$ to glue together the $R'_i$'s in order to get the surfaces $Y^0_g$ and $Y_g=E_g$.

Let $\Theta:\coprod R_i \rightarrow \coprod R'_i$ be the map sending each $R_i$ on the corresponding $R'_i$ via $\Psi_{f,g}$. It clearly induces a $K$-conformal diffeomorphism from $Y^0_f$ to $Y^0_g$. Proposition \ref{prop_confsad} says precisely that this diffeomorphism extends into a $K$-conformal diffeomorphism $\Phi$ between $(E_f, K_f)$ and $(E_g, K_g)$.  It follows from the proof of Lemma \ref{lem_gen} that $\Phi$ conjugates the generic reflections of $E_f$ to the generic reflections of $E_g$.

Reciprocally, if $E_f$ and $E_g$ are $K$-conformal then so are $R_f$ and $R_g$, therefore there exists $a\in \R^*$ such that $f\partial_t$ and $a g\partial_t$ are diffeomorphic (again by Proposition \ref{prop_rib}).

Let $\xi\in  \mathrm{Is}^\pm(E_f, K_f)$ and let  $R$ be a ribbon of $E_f$.  The expression of $\xi$ using "ribbon-coordinates" on $R$ and $\xi(R)$ gives an isometry of $R_f$.  According to \cite[Proposition 4.1]{BetM}, this isometry of $R_f$ does not depend on the choice of $R$ and, because $\mathrm{Is}_{\rm gen}(E_f)/\mathrm{Is}_{\rm gen}^0(E_f)$ acts simply transitively  on the set of  ribbons this correspondence induces an isomorphism between $\mathrm{Is}^\pm(E_f,K_f)/(\mathrm{Is}_{\rm gen}(E_f)/\mathrm{Is}_{\rm gen}^0(E_f))$ and 
$\mathrm{Is}^\pm(R_f,K_f)$,  the group of isometry of $R_f$ sending $K_f$ on $\pm K_f$.   Moreover, elements of  $\mathrm{Is}^\pm(R_f,K_f)$ come from symmetries of $f$ or from the flow $K_f$.
Consequently, if neither $f$ nor $g$ is periodic then  $\mathrm{Is}_{\rm gen}(E_f)$ and  $\mathrm{Is}_{\rm gen}(E_g)$ have finite index (at most $2$)  in $\mathrm{Is}^\pm(E_f, K_f)$ and $\mrm{Is}^\pm(E_g, K_g)$.

Assume now that $f$ and $g$ are periodic and that $X_{f,P}$ is diffeomorphic to $X_{g,Q}$. 
It means that there exists  $\varphi:\R\rightarrow \R$ such that $\varphi_*(f(t)\partial_t)=g(t)\partial_t$ and that $\varphi(y+P)=\varphi(y)+Q$. 
 We can start over the former construction with $\Psi_{f,g}(x,y)=(x,\varphi(y))$ and  construct new maps $\Theta$ and $\Phi$. 
Let  $\widetilde \tau_f:\coprod R_i\rightarrow \coprod R_i, (x,y)_i\mapsto  (x,y+P)_{\mu(i)}$ and  $\widetilde \tau_g:\coprod R'_i\rightarrow \coprod R'_i, (x,y)_i\mapsto  (x,y+Q)_{\mu(i)}$. 
The (new) map $\Theta$  conjugates these maps  and therefore $\Phi$ conjugates the isometries  $\tau_f$ and $\tau_g$ induced by them (see proof of Lemma \ref{lem_gen}).

The flows of $K_f$ and the map $(x,y)\mapsto (x,y+P)$ (respectively the flow of $K_g$ and   $(x,y)\mapsto (x,y+Q)$)  generate a finite index subgroups of $\mathrm{Is}^\pm(R_f,K_f)$ (resp.\  $\mathrm{Is}^\pm(R_g,K_g)$). 
The subgroup of $\mathrm{Is}^\pm(E_f, K_f)$ (resp.\ $\mrm{Is}^\pm(E_g, K_g)$) generated by  $\mathrm{Is}_{\rm gen}(E_f)$ and  $\tau_f$ (resp.\ $\mathrm{Is}_{\rm gen}(E_g)$ and $\tau_g$) have therefore  finite index in $\mathrm{Is}^\pm(E_f, K_f)$ (resp. $\mrm{Is}^\pm(E_g, K_g)$). 
\end{proof}
\begin{cor}\label{cor_ona}
Let $f$ and $g$ be two smooth non constant functions such that $f\partial_t$ is diffeomorphic to $g\partial_t$.
If $f$ and $g$ are both non periodic or both periodic and if 
there exists $(P,Q)\in \mathrm{Per}(f)\times \mathrm{Per}(g)$ such that $X_{f,P}$ and $X_{g,Q}$ are diffeomorphic,
then any  Lorentzian surface $\Sigma$  locally modeled on $E_{f}$ admits a finite cover that is K-conformal to a surface $\Sigma'$   locally modeled on $E_{g}$.
\end{cor}
\begin{proof} 
Let $\widetilde \Sigma$ be the universal cover of $\Sigma$ and $\Gamma$ be the fundamental group of $\Sigma$. According to \cite[Lemme 3.15]{BetM}, there exists local isometry (the so-called  developing map ) $\mathcal D:\widetilde \Sigma\rightarrow E_{f}$ and a group homomorphism $\rho: \Gamma\rightarrow \mrm{Is}^\pm(E_{f},K_{f})$ such that  for any $\nu\in \Gamma$, $\mathcal D\circ \nu=\rho(\nu)\circ \mathcal D$. 

Theorem \ref{theo_conf} says that there exists a conformal diffeomorphism $\Phi: E_{f}\rightarrow E_{g}$ such that $(\Phi \mathrm{Is}^\pm(E_{f},K_{f})\Phi\ms)$ and $ \mrm{Is}^\pm(E_{g},K_{g})$ are commensurable.
Therefore, replacing possibly $\Sigma$ by a finite cover of itself, we can assume that $\Phi\rho(\Gamma)\Phi\ms\subset  \mrm{Is}^\pm(E_{g},K_{g})$. We denote by $\widetilde \Sigma'$ the surface $\widetilde \Sigma$  endowed with the pull-back by $\Phi\circ \mathcal D$ of the metric of $E_{g}$, it  is  $K$-conformal to $\widetilde \Sigma$ and  invariant by the action of $\Gamma$. Thus, the surface $\Sigma'=\widetilde \Sigma'/\Gamma$ has the desired properties.
\end{proof}
\begin{rema} 
As noted at the end of section \ref{sect_ribb}, to any diffeomorphism  $\varphi\neq \mathrm{id}$ sending $f\partial_t$ (respectively  $X_{f,P}$) on itself, we can associate  a non trivial $K$-conformal diffeomorphism $\psi_{f,f}$ of $R_f$.  Repeating  the proof of Theorem \ref{theo_conf} with this map instead of $\Psi_{f,g}$   provides us a non trivial $K$-conformal diffeomorphisms of $E_f$ centralizing $\mathrm{Is}_{\rm gen}(E_f)$ (respectively a finite index  subgroup of $\mathrm{Is}^\pm(E_f, K_f)$). Similarly, the proof of Corollary \ref{cor_ona} shows that if $\varphi\neq \mathrm{id}$ sends $f\partial_t$ ($X_{f,P}$  when $f$ is periodic) on itself, then $\varphi$ induces a non trivial  $K$-conformal transformations of a finite cover of any surface locally modeled on $E_f$.
\end{rema}
When $\Sigma$ is locally modeled on $E_f$, compact and non flat the space of leave of its Killing field $K_f$ is  a circle of length $P_\Sigma$. The function $f$ being given by $\langle K,K\rangle$ it is naturally $P_\Sigma$-periodic, $P_\Sigma$ 
  may not be the smallest positive  period of $f$ (as $\widetilde \Sigma$ may have smaller quotients) again see \cite{BetM} for details. 
In this case, we denote by $X_\Sigma$ the vector field on $\R/P_\Sigma\Z$ induced by $f\partial_t$.
\subsection{Conformal geometry of tori}
Corollary \ref{cor_ona} implies one half of Theorem \ref{cordestores}, the other one follows from the two following propositions.
\begin{prop}\label{prop_TT'}
Let $T$ and $T'$ be non flat Lorentzian tori with Killing fields. If $T$ and $T'$ are $K$-conformal then the vector fields 
$X_T$ is diffeomorphic to a multiple of~$X_{T'}$.
\end{prop}
\begin{proof}
Let $\Psi:T\rightarrow T'$ be a $K$-conformal diffeomorphism. 
We denote by  $\widehat T$ and  $\widehat T'$ the holonomy coverings of $T$ and $T'$  ie  their cyclic coverings whose Killing fields  have no closed obits. 
The complement of the lightlike orbits of $\widehat K$ (a Killing field of $\widehat T$)  is a (possibly infinite) union of strips. A finite, even number of them are of "type 2" ie are bounded by lightlike geodesics belonging to distinct lightlike foliations. We denote them $S_1,\dots, S_{2k}$.  Note that $k$ is a conformal invariant.

If $k=0$, then there exists a lightlike geodesic cuting every orbit of $\widehat K$. It means that  the universal cover of $T$ (resp. $T'$) is (isometric to) a ribbon $R_f$ (resp. $R_g$). 
Consequently,  $\widehat T\simeq R_f/\langle \tau \rangle$ (resp.  $\widehat T\simeq R_g/\langle \tau' \rangle$) where $\tau$ (resp. $\tau'$)  is the isometry  that reads $(x,y)\mapsto (x+\delta  ,y+P)$ (resp.  $(x,y)\mapsto (x+ \delta' ,y+P')$) where $P>0$ (resp  $P'>0$) is the natural period of $f$ (resp. $g$)  and $\delta\in \R$ (resp $\delta'$) is a twist parameter. The map $\Psi$ lifts into a map $\widetilde \Psi: R_f\rightarrow R_g$ that reads $(x,y)\mapsto (ax,\varphi(y))$, for some $a\neq 0$,  and satisfies $\widetilde \Psi\circ \tau=\tau'\circ \widetilde \Psi$. Thus we have $\varphi(y+P)=\varphi(y)+P'$ therefore $\varphi$ induces a diffeomorphism between $\R/P\Z$ and $\R/P'\Z$  that sends $X_T$ on $aX_{T'}$ according to Proposition \ref{prop_rib}.

When $k>0$, it follows that $\widehat T$ is a union of $2k$ ribbons $R_{f_1},\dots , R_{f_{2k}}$ such that the first strip of $R_{f_i}$ is $S_i$ and the last one is $S_{i+1}$ (of course $S_{2k+1}=S_1$). For any $1\leq i \leq 2k$ $f_i$ and $f_{i+1}$ coincide on $S_{i+1}$ so that they glue together in the periodic function $f$. Hence they can be seen as "sub-ribbons" of $R_f$. The ribbons  $R_{f_i}$ and $R_{f_{i+1}}$  are glued together thanks to a generic reflection $\sigma_{i+1}$ (it must be an isometry of $S_{i+1}$ and it must swap the lightlike foliations). 

Similarly $\widehat T'$ is a union of $2k$ ribbons denoted $R_{g_1},\dots ,R_{g_{2k}}$ glued on the strips $S'_i$ thanks to generic reflections $\sigma'_i$. We denote by $\Psi_i:R_{f_i}\rightarrow R_{g_i}$ the maps obtained by restriction of the lift of the $K$-conformal map $\Psi$.

According to Proposition \ref{prop_rib}, $\Psi_i$ reads $(x,y)\mapsto (ax+t_i,\varphi_i(y))$ where $\varphi_i$ is a diffeomorphism sending $f_i\partial_t$ on $ag_i\partial_t$, $a\neq 0$ does not depend on $i$. 
 Moreover, for any $(x,y)\in S_{i}$, we must have $ \Psi_{i-1}(x,y)=\sigma'_{i} \circ \Psi_{i}\circ \sigma_{i}(x,y)$. But $\Psi_{i}$ is $K$-conformal therefore there exits another generic reflection $\sigma''_i$ such $\Psi_{i}\circ \sigma_{i}=\sigma''_i\circ \Psi_i$. But $\sigma'_i\circ \sigma''_i$ is a horizontal translation therefore $\varphi_i$ and $\varphi_{i-1}$ coincide on $S_{i}$. It means that the $ \varphi_i$ glue together in a diffeomorphism of the circle that sends $X_T$ on $aX_{T'}$.
\end{proof}
\begin{rema}\label{rem_u}
If in the statement of Proposition \ref{prop_TT'} we add that $T$ and $T'$ are both modeled on the same $E_f$, then $T$ and $T'$ are isometric. It is clearly sufficient  to prove that $\widehat T$ and $\widehat T'$ are isometric. When $k=0$, it follows from the fact that $\tau=\tau'$. When $k$ is positive, it emerges from the proof above that there exists a  map $\Theta:\widehat T\rightarrow \widehat T'$ sending each $R_{f_i}$ on itself by $(x,y)\mapsto (ax+t_i, ay)$ (the compatibility in only encoded in the $t_i$'s). But $|a|$ must be $1$, as the period of a flow is invariant by conjugacy, therefore $\Theta$ is an isometry.
\end{rema}
\begin{prop}  \label{prop>K}
If two non conformally flat Lorentzian tori with Killing fields  $T$ and $T'$  are conformal then they are  $K$-conformal.
\end{prop}
\begin{proof}We consider an $S^1$-action preserving the lightlike foliations  of $T$, denoted  $\mathcal F$ and $\mathcal F'$. 
We will prove that this action coincide, up to diffeomorphism, with  the action induced by the Killing field of $T$.
By hypothesis, at least one  of these foliations is not linear, we assume it is $\mathcal F$.
 We lift $\mathcal F$ and $\mathcal F'$ to the holonomy covering $\widehat T$  and consider their spaces of leaves $\mathcal L$ and $\mathcal L'$.  These spaces have a natural, possibly non Hausdorff (it depends on the presence of Reeb components), $1$-dimensional manifold structure (see \cite{HetR}).
We denote by $\gamma$ a generator of the cyclic group $\pi_1(T)/\pi_1(\widehat T)$. Our $S^1$-action  on $T$ preserving $\mathcal F$ and $\mathcal F'$ lifts in a $\R$-action on $\widehat T$ that preserves the lifted foliations and therefore induces  a smooth $\R$-action on $\mathcal L$ and $\mathcal L'$.  The fact that the action of $\bar 1$ (seeing $S^1$ as $\R/\Z$) on $T$ is trivial implies that the action of $1$ on $\widehat T$ and therefore on $\mathcal L$ and $\mathcal L'$ coincide with the action of $\gamma$ (or $\gamma\ms$).  Clearly, the $\R$-action on $\widehat T$ is determined by the $\R$-actions on $\mathcal L$ and $\mathcal L'$.

 The foliation  $\mathcal F$ is not linear and is invariant by a $S^1$-action, therefore it  has a proper subset of compact leaves.  The fixed points of the action of $\gamma$ on $\mathcal L$  correspond to the compact leaves of $\mathcal F$ and therefore to the  points fixed by the $\R$-action. 
  These fixed points delimit half-lines on which the action of $\gamma$ has no fixed points and therefore satisfies the hypotheses of Szekeres' Theorem (see \cite{Szek} or \cite[Theorem 4.1.11]{Nav}) which says that the $\R$-action on each of these lines, and therefore on  $\mathcal L$ by the above remark,  is determined by the action of $\gamma$.  If $\mathcal F'$ is also not linear then the same is true for the action of $\R$ on $\mathcal L'$. Hence, in this case, there exists only one  $\R$-action on $\widetilde T$  preserving the lifted foliations: the one induced by $K$.

 If $\mathcal F'$  is  linear, then, using the notations from the proof of Proposition \ref{prop>K}, $k=0$, ie  $\widehat T$ is a quotient of a ribbon. The space  $\mathcal L'$  is either diffeomorphic to $\R$ 
or to $S^1$and  the $\R$-action on it has no fixed points. 
We can assume that $\gamma$ acts on $\mathcal L'$ as a translation. The $\R$-action fixing no point there exists a diffeomorphism $\varphi$ of $\mathcal L'$ fixing a point,  conjugating the action to an action by translations and commuting with the action of $\gamma$ (if $\mathcal L'\simeq S^1$ it comes from the invariance by conjugacy  of the rotation number of $\gamma$ and if $\mathcal L'\simeq \R$ from the fact that any pair of non trivial translations are conjugated). Moreover, there exists a diffeomorphism $\Phi$ of $\widehat T$ that induces $\varphi$ on $\mathcal L'$ and the identity on $\mathcal L$. Indeed, any leaf of $\widehat {\mathcal F}$ can be seen as a covering of $\mathcal L'$. We endow each leaf of $\mathcal F$ by the lift of $\varphi$ that has fixed points. Taking coordinates on which $\widehat {\mathcal F}$ and $\widehat {\mathcal F}'$ read as  product foliations, it is clear that  the map $\Phi$ obtained is smooth. Clearly $\Phi$ commutes with $\gamma$ and therefore induces a diffeomorphism of $T$ that conjugates our $S^1$-action with the one given by $K$.
\end{proof}

Hyperbolic vector fields on the line are rather easy to compare. Indeed, it is enough to compare their zeros, see \cite{BelTrak} for a proof, more precisely:
\begin{theo}\label{theo_classR}
Let $f$ and $g$ be two smooth functions on $\R$ and let  $\dots <z_i<z_{i+1}<\dots$ and $\dots  <z'_i<z'_{i+1}<\dots, i\in \Z$ be their sets of zeros. If for any $i\in \Z$ $f'(z_i)=g'(z'_i)\neq 0$  then the vector fields $f(t)\partial_t$ and $g(t)\partial_t$  are difffeomorphic .
\end{theo}
\begin{defi}\begin{itemize}
\item
We will say that a function  $f:\R\rightarrow \R$  satisfies Mehidi's condition if there exists $\lambda>0$ such that for any $z\in \R$, $f(z)=0$ implies $|f'(z)|=\lambda$. By extension we will say that a Lorentzian torus with a Killing field satifies Mehidi's condition if it is locally modeled on $E_f$ and $f$ satisfies Mehidi's condition.
\item
A Lorentzian torus is Reeb if its lightlike foliations are unions of Reeb components.
\end{itemize}
\end{defi}
\begin{theo}[Mehidi \cite{LM}]\label{thM}
Let $T$ be a non flat Lorentzian torus  locally modeled on $E_f$. If $T$ has no conjugate points and if the zeros of $f$ are simple (or equivalently if the closed lightlike geodesic are not complete)  then  $T$ is Reeb and $f$ satisfies Mehidi's condition. Moreover the reciprocal is true for $f$ close enough from  $4\sin$.
\end{theo}
The case $f(t)=4\sin(t)$ corresponds to the Clifton-Pohl tori, ie tori having the same universal cover as $T_{CP}:=\big (\R^2\smallsetminus \{0\},{\displaystyle \frac{2dxdy}{x^2+y^2}}\big)/((x,y)\sim 2(x,y))$ (see \cite[Exemple 3.27]{BetM}). It was proven in  \cite{BetM1} that  Clifton-Pohl tori do not have conjugate points. 
\begin{cor}
A torus with a Killing field  is Reeb and satisfies  Mehidi's condition  if and only if its  universal cover is $K$-conformal to the universal cover of the Clifton-Pohl torus. 
\end{cor}
\begin{proof}
Let $T$ be a Reeb  torus locally modeled on $E_f$ satisfying Mehidi's condition. Replacing possibly $f$ by another element of $[[f]]$, we can assume that $|f'( z)|=4$ for any $z$  zero of $f$ (as it is possible to replace $f(y)$ by $a^2f(y/a)$). It follow from Theorem \ref{theo_classR} that $f(t)\partial_t$ and  $4\sin(t)\partial_t$ are diffeomorphic and therefore $E_f$ is $K$-conformal to $E_{4\sin}$ (the extension containing  the universal cover of the Clifton-Pohl torus). 

According to \cite[Théorème 3.25 and Proposition 4.35]{BetM}, $\widetilde T$,  the universal cover of  $T$, is isometric to the saturation under the flow of $K_f$ of a maximal broken lightlike geodesic $\gamma$ of $E_f$ that has at most one breaking point per strip  and never cuts  two distinct separatrices of a saddle point (or equivalently that never cuts twice an integral curve of $K_f$). It follows from  \cite[Lemme 5.17]{BetM} that if the lightlike foliations of $T$ are unions of Reeb components then~$\gamma$ has one breaking point on each strip. This condition characterizes $\widetilde T$ up to isometry of $E_f$ and of course it also  characterizes $\widetilde T_{CP}$, the universal cover of $T_{CP}$, inside $E_{4\sin}$. Therefore the $K$-conformal map between $E_{4\sin}$ to $E_f$ sends any copy of $\widetilde T_{CP}$ on a copy of $\widetilde T$. 

Reciprocally, if $\widetilde T$ is $K$-conformal to $\widetilde T_{CP}$ then \cite[Lemme 5.17]{BetM} tells us that $T$ is also Reeb. Moreover, on any domino of $\widetilde T$ the metric reads $\zeta(y)(4\sin(y)dx^2+2dxdy)$ for some positive function $\zeta$.  The function $f$ is then $4(\zeta\circ Z)(\sin\circ Z)$ where $Z$ is a solution of $y'=\frac 1{\zeta(y)}$ and at $z$ any zero of $f$ we have $|f'(0)=4|$. Thus, $T$ satisfies Mehidi's condition.
\end{proof}
To classify hyperbolic vector fields on the circle, another invariant is needed. For Lorentzian tori it means that being $K$-conformal or having $K$-conformal universal cover are two different things. 
\begin{defi}[see \cite{Hitch} and \cite{Byk}]
Let $X$ be a hyperbolic  vector field on the circle $\R/P\Z$ ie a vector field  induced by a  vector field of the line $f(t)\partial_t$,   $f$ being $P$-periodic and  having only  simple zeros. Let $z_1<\dots < z_n$ be the zeros of $f$ in $[0,P[$. We define (with $z_{n+1}=z_1+P$)
$$\mu(X)=\lim_{\varepsilon\rightarrow 0}\sum_{i=1}^n\int_{z_i+\varepsilon}^{z_{i+1}-\varepsilon}\frac 1f \in \R $$
\end{defi}
\begin{defi}
Let $X_{f,P}$  be the hyperbolic vector field on $\R/P\Z$ induced by a $P$-periodic map $f$. The list of invariants of $X_{f,P}$ is the following
\begin{itemize}
\item the number $n_f$ of zeros of $X_{f,P}$ (ie of $f$ on $[0,P[$),
\item the values of the  $\lambda_{f,i}:=f'(z_i)$, where $z_i$ is ith zero of $f$,
\item the global invariant $\mu(X_{f,P})$ (it is indeed invariant under circle diffeomorphism).
\end{itemize}
\end{defi}
\begin{theo}[Hyperbolic vector fields on the circle -- Hitchin \cite{Hitch}, Bykov \cite{Byk}--]\label{HB}
Two hyperbolic vector fields on $S^1$ are diffeomorphic if and only if the have the same list of invariants. 
\end{theo}
Let  $b\in  ]-1,1[$, and  $f_b:\R\rightarrow \R,b(y)\mapsto  \sin(y)(1+ b \sin(y))$. It is proven in $\cite{Byk}$ that when $b$ runs through $]-1,1[$ then $\mu(X_{f_b,2\pi})$ (and therefore $\mu(X_{f_b,2k\pi})$, for any $k\in \Z^*$) runs through all $\R$.  It follows that if $f$ is $P$-periodic and satisfies Mehidi's condition then there exists $b\in ]-1,1[$ and $k\in \Z^*$ such that $X_{f,P}$ is diffeomorphic to $\lambda X_{f_b,2k\pi}$. Hence,  Theorem  \ref{cordestores} implies the following corollary (uniqueness of the torus actually follows from Remark~\ref{rem_u}).
\begin{cor} If $T$ is a torus with a Killing field satisfying Mehidi's condition, then there exists a unique $b\in ]-1,1[$ and a unique torus locally modeled on $E_{f_b}$ and  $K$-conformal to~$T$.
\end{cor}
\begin{examples}
\begin{itemize}
\item
The quadratic variations of Clifton-Pohl tori are the tori covered by  $(\R^2\smallsetminus\{0\},\frac{2dxdy}{Q(x,y)})$ where $Q$ is a positive definite quadratic form. According to  \cite[Theorem 4.20]{LM} they do not have conjugate points. They  are clearly all $K$-conformal to a Clifton-Pohl tori (no need to compute  $\mu$).
\item It follows directly form \cite[Theorem 4.20]{LM} that if $|b|<1/8$ (because of the property $f'f'''\leq 0$) then  Reeb tori localy modeled on $E_{f_b}$ do not have conjugate points. But they are in the $K$-conformal class of a Clifton-Pohl torus if and only if  $b=0$. It is not clear to the author, if there exists tori without conjugate points in the $K$-conformal classes given by $|b|\geq 1/8$.
\end{itemize}
\end{examples}

\bigskip
\begin{tabular}{ll}
 Address: & Univ. Bordeaux, IMB, UMR 5251, F-33400 Talence, France\\
& CNRS, IMB, UMR 5251, F-33400 Talence, France\\
\\
E-mail:&{\tt pierre.mounoud@math.u-bordeaux.fr}
\end{tabular}

\end{document}